\documentclass[runningheads]{llncs}
\usepackage{appendix}
\usepackage[utf8]{inputenc}
\usepackage[T1]{fontenc}
\usepackage{microtype}
\usepackage{mathtools}
\usepackage{amssymb}
\usepackage{booktabs}
\usepackage{array}
\usepackage{tabularx}
\usepackage{longtable}
\newcolumntype{L}{>{\raggedright\arraybackslash}X} 
\usepackage{enumitem}
\usepackage{csquotes}
\usepackage[scaled=0.9]{zi4}
\usepackage{xcolor}
\usepackage{tikz,graphicx}
\usepackage{pgfplots}
\pgfplotsset{compat=1.18}
\usetikzlibrary{arrows.meta,positioning}
\usepackage{varioref}
\usepackage[colorlinks=true,linkcolor=blue!70!black,citecolor=green!50!black,urlcolor=purple!70!black]{hyperref}
\usepackage{cleveref}
\usepackage[backend=biber,style=numeric-comp,sorting=nyt,backref=true]{biblatex}
\newcommand{\Q}{\mathbb{Q}}
\newcommand{\Z}{\mathbb{Z}}
\newcommand{\R}{\mathbb{R}}

\newcommand{\Gal}{\operatorname{Gal}}
\newcommand{\Frob}{\operatorname{Frob}}

\newcommand{\Res}{\operatorname{Res}}

\newcommand{\todo}[1]{}

\begin{document}

\title{On the Algebraic Complexity of Optimal\\Polynomial Approximation Constants}
\titlerunning{Algebraic Complexity of Optimal Polynomial Approximation Constants}
\authorrunning{Filipović et al.}

\author{Filip Filipović\inst{1,2} \and Rémi Géraud-Stewart\inst{3}\thanks{Work done in 2016 while at ENS.} \and\\
David Naccache\inst{1,3} \and Aleksa Veličković\inst{3,4} }
\institute{
Fakultet inženjerskih nauka, Univerziteta u Kragujevcu\\
6 Sestre Janjić, 34000 Kragujevac, Serbia\\
\email{\url{david@kg.ac.rs}}
\and
BioIRC\\
Prvoslava Stojanovića 6, 34000 Kragujevac, Serbia\\
\email{\url{filip.filipovic@uni.kg.ac.rs}}
\and
  DIENS, ENS, CNRS, PSL University\\
  45 rue d'Ulm, 75230, Paris \textsc{cedex} 05, France\\
  \email{\url{given\_name.family\_name@ens.fr}}
  \and
Be-YS Research\\
10 Boulevard Haussmann, 75009, Paris, France\\
  \email{\url{aleksa.velickovic@be-ys.com}}
}

\maketitle

\begin{abstract}
We investigate the algebraic nature of constants arising from Chebyshev equiripple (minimax) polynomial approximation of $L_p$ norms on $[0,1]$. For the Euclidean case $\sqrt{1+t^2}$, we compute equiripple solutions from degree~1 through~8 and determine exact minimal polynomials and Galois groups for degrees~1 and~2 in both absolute and relative error formulations. We find a sharp phase transition: the degree-1 constants are solvable by radicals (Galois groups $C_4$, $D_4$), while the degree-2 constants provably are not (Galois groups $S_{12}$, $S_{10}\times C_2$). We explain this transition by a structural dichotomy between decoupling and coupling of critical points, and extend the analysis to $L_3$ norms, where the minimal polynomial degree jumps to~246. A general impossibility result follows from Hilbert's irreducibility theorem. We also develop a theory of piecewise equiripple approximation with jointly optimized breakpoints, proving that each doubling of the number of subintervals gains $n+1$ bits of accuracy at no additional arithmetic cost. These results establish a previously unobserved connection between Chebyshev approximation theory and the non-solvability of algebraic equations.
\end{abstract}

\section{Introduction}\label{sec:intro}

\subsection{The approximation problem}
Computing $\sqrt{x^2+y^2}$ is expensive: it requires a square root, which in hardware or embedded software costs significantly more than additions and multiplications. A classical trick, known as the \emph{alpha-max-plus-beta-min} algorithm, replaces the exact norm by the linear combination
\[
  g(x,y) = \alpha\cdot\max(|x|,|y|) + \beta\cdot\min(|x|,|y|)\,,
\]
which requires only a comparison and two multiplications by constants. This folk algorithm has been in use in digital signal processing since at least the early 1990s~\cite{paeth1990,frerking1994}, was published in textbooks~\cite{lyons2011}, patented~\cite{patent1995}, and its optimal coefficients tabulated to 12~decimal places~\cite{griffin1999}.

The natural mathematical question is: \emph{what are the best possible constants $\alpha$ and $\beta$?} More precisely, assuming $x\ge y\ge 0$ and setting $t=y/x\in[0,1]$, we seek the polynomial $h(t)$ of degree~$n$ that minimizes the maximum error $\|h - f\|_\infty$ on $[0,1]$, where $f(t)=\sqrt{1+t^2}$.

By the classical theorem of Chebyshev (1854), such an optimal polynomial exists, is unique, and is characterized by the \emph{equioscillation} (or \emph{equiripple}) property: the error $d(t)=h(t)-f(t)$ attains its maximum absolute value with alternating signs at exactly $n+2$ points in $[0,1]$. We take these as our starting point.

\paragraph{The degree-1 case admits a closed form.}
For $n=1$ (linear approximation), the optimal coefficients have been known numerically since the 1990s: $\alpha\approx 0.9604$, $\beta\approx 0.3978$, achieving a maximum relative error of about~3.96\% (equivalently, a maximum absolute error of~$4.49\times 10^{-2}$). A trigonometric closed form $\alpha = 2\cos(\pi/8)/(1+\cos(\pi/8))$ appears on Wikipedia~\cite{wikipedia_amabm} (added in 2006 without citation or derivation); see also the survey by Celebi et al.~\cite{celebi2011}. We show in Section~\ref{sec:deg1} that $\alpha$ satisfies the polynomial $x^4+24x^3-8x^2-32x+16$ with Galois group~$C_4$, which is expressible by nested radicals. The derivation is straightforward because the equiripple system \emph{decouples}: the location of the critical point can be determined independently of the error level.

\paragraph{A degree-2 improvement.}
A simple improvement adds one rational correction term:
\begin{equation}\label{eq:rational_correction}
  h(x,y) = ax + by + c\,\frac{y(x-y)}{x}\,,
\end{equation}
which, after the substitution $t=y/x$, becomes a degree-2 polynomial in~$t$. At the cost of one extra multiplication and one division, the error drops tenfold, from~4\% to~0.4\%. We ask: what is the exact algebraic nature of the optimal constants $a$, $b$, $c$?

The answer can be worked out explicitly:
\begin{align*}
a &= \frac{1+s-k \left(b+ c (1-k) \right)}{2} \\ 
b &= \frac{k^3 + (2k-1)\left(1 - s \sqrt{2}\right)}{(k-1)^2 s} \\
c &= \frac{1+k-s\sqrt{2}}{(k-1)^2 s}
\end{align*}
where $s=\sqrt{k^2+1}$, and $k$ is some constant 
\[
k = 0.25924696053627332706027906015867463059...
\]
whose value we will discuss further below.
The improvement over degree-1 approximation is illustrated in \Cref{fig:err}.
Note that once the constants $a, b, c$ are computed, evaluating \Cref{eq:rational_correction} requires only 4 multiplications and one division, which don't have cross-dependencies and can thus be executed concurrently.

\begin{figure}[h]
    \centering
    \includegraphics[width=0.5\textwidth]{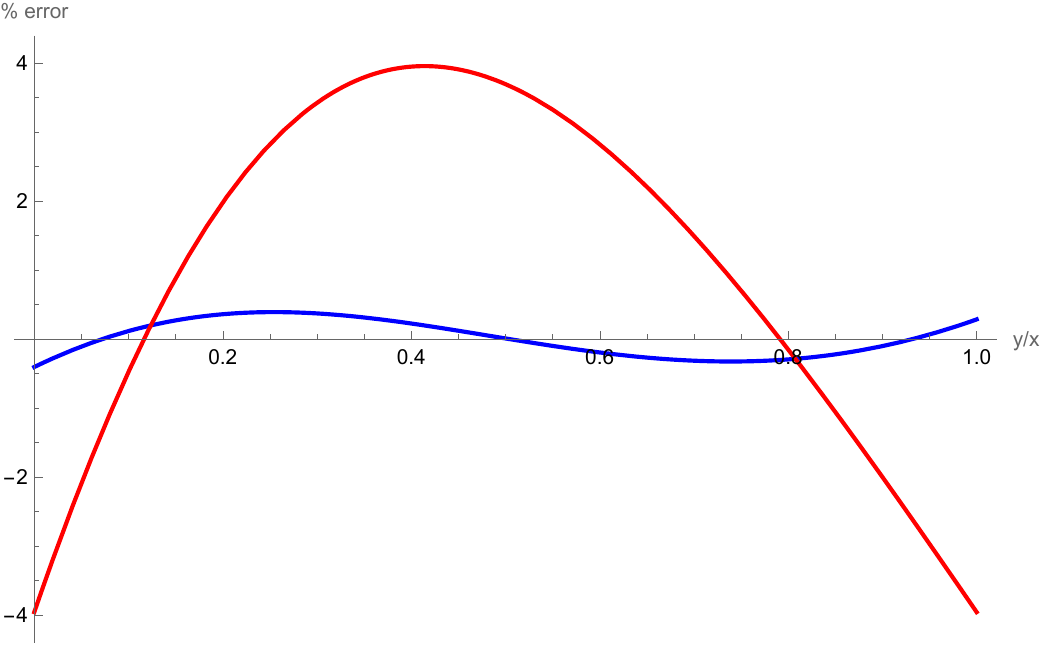}
    \caption{The relative error for degree-1 approximation (\textcolor{red}{$g(x,y)$}) versus degree-2 approximation (\textcolor{blue}{$h(x,y)$}), in terms of $t = y/x$.}
    \label{fig:err}
\end{figure}

The optimal parameter $k$ is in fact the first interior equiripple point. We show it is algebraic of degree~20, with Galois group $S_{10}\times C_2$ of order~$7{,}257{,}600$. Since $S_{10}$ is not a solvable group (it contains the simple group~$A_5$), the parameter $k$ has no closed-form expression in terms of elementary operations.
Note the sharp contrast with the degree-1 case, where the Galois group has order~4 and we have a closed form formula for optimal parameters.

\subsection{Overview}
We prove this result in two different formulations (absolute and relative error), extend to $L_p$ norms where the algebraic complexity grows rapidly (degree~246 for $L_3$), establish a general impossibility theorem via Hilbert's irreducibility theorem, and explain the structural mechanism (coupling of critical points) that drives the transition. We also investigate piecewise approximation with jointly optimized breakpoints, showing that subdivision gains $n+1$ bits per doubling of subintervals at zero arithmetic cost. Numerical data for equiripple solutions from degree~1 through~8 completes the picture.

\medskip
\noindent Our main contributions are:
\begin{enumerate}
  \item Two explicit irreducible polynomials (of degrees~20 and~12) whose roots are the degree-2 equiripple constants for $\sqrt{1+t^2}$ in the absolute and relative formulations respectively, along with proofs that these polynomials have Galois groups $S_{10}\times C_2$ and $S_{12}$ (using Jordan's theorem applied to Frobenius elements at the primes $p=17$ and $p=257$, see \Cref{sec:main}).
  \item A structural explanation for why the degree-1 constants are solvable (the equiripple system decouples) while the degree-2 constants are not (the system is irreducibly coupled) (\Cref{sec:transition}).
  \item A computation showing that the analogous constant for the $L_3$ norm has minimal polynomial of degree~246, indicating rapid growth of algebraic complexity with the norm parameter (\Cref{sec:lp}).
  \item A proof, via Hilbert's irreducibility theorem, that the non-solvability phenomenon is generic: it holds for almost all target functions in a one-parameter family (\Cref{sec:general}).
  \item An analysis of the arithmetic structure of the degree-20 polynomial via its Newton polygons at the primes~2 and~7 (\Cref{sec:arithmetic}).
  \item Numerical equiripple solutions for polynomial degrees~1 through~8, with an analysis of the asymptotic efficiency governed by the Bernstein ellipse parameter $\rho\approx 4.612$ (\Cref{sec:efficiency}).
  \item A theory of piecewise equiripple approximation with jointly optimized breakpoints: we prove that each doubling of subintervals gains $n+1$ bits of accuracy at zero additional arithmetic cost, and show that degree-1 linear approximation with $k=4$ subintervals already surpasses the single-interval degree-2 result (\Cref{sec:piecewise}).
\end{enumerate}

\section{Solvability of degree-1 approximation}\label{sec:deg1}

This section recovers the degree-1 formula in the mathematical scaffold that we use throughout this work. 

\subsection{Relative error optimization}
We seek $h(t)=\alpha+\beta t$ minimizing $\max_{t\in[0,1]}|e(t)|$ where $e(t) = h(t)/\sqrt{1+t^2}-1$ is the relative error. By Chebyshev's theorem, the optimal approximation equioscillates at $n+2=3$ points. Let these be $t=0$, $t=t_1$ (interior), and $t=1$.

The critical point equation $e'(t_1)=0$ reduces (after clearing denominators) to $\beta - \alpha t_1 = 0$, giving
$
  t_1 = \beta/\alpha
$.
Crucially, \emph{the critical point depends only on the ratio of coefficients}.

The equiripple condition $e(0) = e(1)$ gives $\alpha - 1 = (\alpha+\beta)/\sqrt{2} - 1$, i.e., $\alpha = (\alpha+\beta)/\sqrt{2}$, hence $\beta = \alpha(\sqrt{2}-1)$. Substituting we get
$
  t_1 = \sqrt{2}-1 \approx 0.4142
$.
Note that this value is \emph{independent of the error level}~$\varepsilon$. In other words, the system has \enquote{decoupled}: we have determined $t_1$ without knowing $\alpha$.

To find $\alpha$, we use the remaining condition $e(t_1)=-e(0)$ (alternating sign). After squaring to eliminate $\sqrt{1+t_1^2}$, then eliminating $\sqrt{2}$, one obtains:

\begin{proposition}\label{prop:deg1}
The optimal coefficient $\alpha$ satisfies the irreducible polynomial
\[
  P_1(x) = x^4+24x^3-8x^2-32x+16 \in\Q[x]
\]
with Galois group~$C_4$ (cyclic of order~4), discriminant~$2^{29}$, and roots
\[
  \alpha = \frac{2\cos(\pi/8)}{1+\cos(\pi/8)} = \frac{2}{1+\sqrt{4-2\sqrt{2}}} \approx 0.9604\,.
\]
\end{proposition}

\begin{proof}
From $e(t_1)=-(\alpha-1)$ and $t_1=\sqrt{2}-1$, we derive $\alpha = 2/(1+\sqrt{4-2\sqrt{2}})$. Setting $w=\sqrt{4-2\sqrt{2}}$, so $w^2=4-2\sqrt{2}$ and $\sqrt{2}=(4-w^2)/2$, we substitute $\alpha = 2/(1+w)$ and square to eliminate~$\sqrt{2}$, obtaining $8\alpha^4 = (3\alpha^2+4\alpha-4)^2$, which expands to $P_1(\alpha)=0$. 

Irreducibility and the Galois group $C_4$ can be verified by a variety of methods at this scale, we rely on PARI/GP (\texttt{polisirreducible} and \texttt{polgalois}).
\end{proof}

\subsection{Absolute equiripple optimization}
For the absolute error formulation, which minimizes $\max_{t\in[0,1]}|h(t)-\sqrt{1+t^2}|$, a similar derivation gives the critical point $t_1=\sqrt{(\sqrt{2}-1)/2}$. This quantity satisfies the irreducible polynomial $4x^4+4x^2-1=0$, whose Galois group is~$D_4$, the dihedral group of order~8. The result is again solvable by radicals.

\subsection{Why the degree-1 system is solvable}
In both formulations, the degree-1 system decouples for the following reasons:
\begin{enumerate}
  \item The derivative condition $e'(t)=0$ determines $t_1$ as a ratio of the polynomial coefficients, with no dependence on the error level.
  \item The endpoint equiripple condition then pins this ratio to a specific algebraic number (namely $\sqrt{2}-1$ in the relative case).
  \item Only after $t_1$ has been determined does one solve for the absolute scale of the coefficients, which involves a single quadratic equation over $\Q(\sqrt{2})$.
\end{enumerate}
The total algebraic complexity is degree~4. 

\section{Non-solvability of degree-2 approximation}\label{sec:main}

\subsection{The equiripple system}
The degree-2 polynomial $h(t) = a+(b+c)t-ct^2$ approximates $f(t)=\sqrt{1+t^2}$ on $[0,1]$. For the absolute equiripple, the error $d(t)=h(t)-f(t)$ alternates between $+\varepsilon$ and $-\varepsilon$ at four points $\{0,t_1,t_2,1\}$. Figure~\ref{fig:equiripple} shows this characteristic oscillation pattern.

\begin{figure}[tbp]
\centering
\begin{tikzpicture}
\begin{axis}[
  width=\linewidth, height=4.5cm,
  xlabel={$t$},
  ylabel={$d(t)\times 10^3$},
  xmin=-0.02, xmax=1.02,
  ymin=-5.5, ymax=5.5,
  xtick={0, 0.2592, 0.7489, 1},
  xticklabels={$0$, $t_1\!=\!k$, $t_2$, $1$},
  ytick={-4.04, 0, 4.04},
  yticklabels={$-\varepsilon$, $0$, $+\varepsilon$},
  axis lines=middle,
  every axis x label/.style={at={(ticklabel* cs:1)}, anchor=west},
  every axis y label/.style={at={(ticklabel* cs:1)}, anchor=south},
  grid=major,
  grid style={gray!15},
  tick label style={font=\footnotesize},
]
\addplot[blue!70!black, thick, smooth, domain=0:1, samples=200]
  {1000*(0.99596 + 0.06644*x + 0.35586*x*(1-x) - sqrt(1+x^2))};
\addplot[red!60!black, densely dashed, thin, domain=0:1] {4.04};
\addplot[red!60!black, densely dashed, thin, domain=0:1] {-4.04};
\node[circle, fill=blue!70!black, inner sep=1.3pt] at (axis cs:0,-4.04) {};
\node[circle, fill=blue!70!black, inner sep=1.3pt] at (axis cs:0.2543,4.04) {};
\node[circle, fill=blue!70!black, inner sep=1.3pt] at (axis cs:0.7490,-4.04) {};
\node[circle, fill=blue!70!black, inner sep=1.3pt] at (axis cs:1,4.04) {};
\end{axis}
\end{tikzpicture}
\caption{The equiripple error $d(t) = h(t) - \sqrt{1+t^2}$ for the optimal degree-2 polynomial on $[0,1]$. The error attains the same magnitude $\varepsilon \approx 0.00404$ with alternating signs at four points. The parameter $k = t_1 \approx 0.2592$ is the first interior alternation point whose minimal polynomial we determine.}
\label{fig:equiripple}
\end{figure}
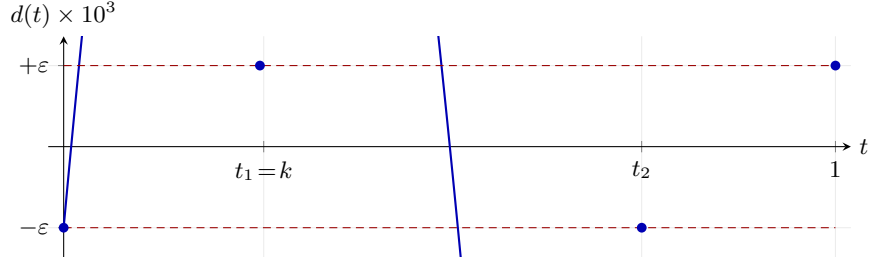

We now derive the polynomial system whose solutions give the equiripple parameters. Writing $h(t) = a + (b+c)t - ct^2$, the four equiripple conditions $d(0)=-\varepsilon$, $d(t_1)=+\varepsilon$, $d(t_2)=-\varepsilon$, $d(1)=+\varepsilon$ read:
\begin{align}
  a - 1 &= -\varepsilon\,,\label{eq:sys1}\\
  a + (b+c)t_1 - ct_1^2 - \sqrt{1+t_1^2} &= +\varepsilon\,,\label{eq:sys2}\\
  a + (b+c)t_2 - ct_2^2 - \sqrt{1+t_2^2} &= -\varepsilon\,,\label{eq:sys3}\\
  a + b - \sqrt{2} &= +\varepsilon\,.\label{eq:sys4}
\end{align}
From~\eqref{eq:sys1} and~\eqref{eq:sys4}: $a=1-\varepsilon$ and $b = \sqrt{2} + \varepsilon - a = \sqrt{2} - 1 + 2\varepsilon$. Thus $a$ and $b$ are determined by~$\varepsilon$, and~\eqref{eq:sys2}--\eqref{eq:sys3} become two equations in three unknowns $(t_1, t_2, c)$ with $\varepsilon$ as a parameter.

The derivative conditions $d'(t_i) = 0$ give:
\begin{equation}\label{eq:deriv}
  (b+c) - 2c\,t_i = \frac{t_i}{\sqrt{1+t_i^2}}\,, \quad i=1,2\,.
\end{equation}
Adding~\eqref{eq:sys1} to~\eqref{eq:sys2} and to~\eqref{eq:sys3} eliminates~$\varepsilon$, and subtracting~\eqref{eq:sys2} from~\eqref{eq:sys4} gives a relation between $t_1$, $t_2$, $c$, and the radicals $\sqrt{1+t_i^2}$.

To obtain a purely polynomial system, we introduce algebraic variables $w_1$, $w_2$, $s$ satisfying the \emph{radical constraints}:
\begin{equation}\label{eq:radical}
  w_1^2 = 1 + t_1^2\,, \quad w_2^2 = 1 + t_2^2\,, \quad s^2 = 2\,,
\end{equation}
and substitute $\sqrt{1+t_i^2}\mapsto w_i$, $\sqrt{2}\mapsto s$. Setting $k\coloneqq t_1$, the resulting polynomial system has 6~equations in 6~unknowns $(k, w_1, t_2, w_2, c, s)$:
\begin{enumerate}
  \item From~\eqref{eq:sys2}$+$\eqref{eq:sys1}: $(b+c)k - ck^2 - w_1 + 1 = 2\varepsilon$, with $b = s - 1 + 2\varepsilon$.
  \item From~\eqref{eq:sys3}$+$\eqref{eq:sys1}: $(b+c)t_2 - ct_2^2 - w_2 + 1 = 0$.
  \item From~\eqref{eq:deriv} at $t_1$: $(b+c) - 2ck = k/w_1$.
  \item From~\eqref{eq:deriv} at $t_2$: $(b+c) - 2ct_2 = t_2/w_2$.
  \item $w_1^2 - 1 - k^2 = 0$.
  \item $w_2^2 - 1 - t_2^2 = 0$.
\end{enumerate}
The constraint $s^2=2$ is adjoined when eliminating~$s$ in the final step.

The crucial structural point: in equations~(3) and~(4), the coefficient~$c$ appears linearly alongside both~$k$ and~$t_2$. Subtracting gives $c = \frac{1}{2}\left(\frac{k/w_1 - t_2/w_2}{k - t_2}\right)$, showing that $c$ depends on \emph{both} critical points simultaneously. Unlike the degree-1 case (Section~\ref{sec:deg1}), no variable can be isolated first: the system is \enquote{coupled}.

\subsection{Resultant elimination}
We use the classical technique of \emph{resultant elimination} to reduce the system to a single univariate polynomial in~$k$. Recall that for two univariate polynomials $f(x)$ of degree~$m$ and $g(x)$ of degree~$n$, the resultant $\Res_x(f,g)\in\Z[\text{other variables}]$ vanishes if and only if $f$ and $g$ share a common root; it has degree at most~$mn$ in the remaining variables.
We eliminate variables sequentially:
\begin{enumerate}
  \item \textbf{Eliminate $c$} (appears linearly): subtract equation~(3) from~(4), solve for~$c$, and substitute back. This is exact (no degree increase from a resultant).
  \item \textbf{Eliminate $w_2$}: equation~(4) gives $w_2$ rationally in terms of $t_2$ and~$c$; substituting into the constraint $w_2^2 = 1 + t_2^2$ yields a polynomial in $(k, w_1, t_2, s)$.
  \item \textbf{Eliminate $t_2$} via $\Res_{t_2}$: the two remaining equations have degrees~3 and~6 in~$t_2$, producing a polynomial of degree~$\le 18$ in $(k, w_1, s)$.
  \item \textbf{Eliminate $w_1$} via $\Res_{w_1}$ with the constraint $w_1^2 - 1 - k^2$: degrees $18\times 2 = 36$ in $(k, s)$. After removing content, effective degree~24 in~$k$.
  \item \textbf{Eliminate $s$} via $\Res_s$ with $s^2 - 2$: degrees $24\times 2 = 48$ in~$k$.
\end{enumerate}
The raw output has degree~160 in~$k$ (accounting for multiplicities from the alternating-sign ambiguities in the radical substitutions). Factoring over~$\Q$ yields:
\[
  R(k) = Q(k)^2 \cdot S_1(k) \cdot S_2(k) \cdots
\]
where $Q(k)$ is the irreducible degree-20 polynomial of Theorem~\ref{thm:absolute} (appearing squared due to the $w_1\mapsto\pm\sqrt{1+k^2}$ ambiguity) and the~$S_i$ are lower-degree spurious factors corresponding to sign-inconsistent solutions of the algebraized system.

\subsection{The absolute equiripple polynomial}

\begin{theorem}\label{thm:absolute}
The first interior alternation point $k$ of the optimal absolute-error degree-2 equiripple polynomial approximation of $\sqrt{1+t^2}$ on $[0,1]$ is an algebraic number of degree~20, root of the irreducible polynomial
\begingroup\small
\begin{align*}
Q(x) ={} & 65536x^{20} + 704512x^{18} - 172032x^{17} + 3338240x^{16} - 517120x^{15}\\
  & + 6763904x^{14} + 126656x^{13} + 6977244x^{12} + 1107728x^{11}\\
  & + 4588444x^{10} + 559320x^9 + 2468577x^8 - 140344x^7\\
  & + 911252x^6 - 62456x^5 + 118178x^4 + 40x^3 - 568x^2 - 1568x - 343.
\end{align*}
\endgroup
The Galois group is $\Gal(Q/\Q) \cong S_{10}\times C_2$, of order $7{,}257{,}600$. In particular, $k$~is not expressible by radicals.
\end{theorem}

\begin{proof}[Proof of the Galois group] ~ 

\textbf{Step 1} (Irreducibility). Verified by PARI/GP (\texttt{polisirreducible}).

\textbf{Step 2} (Factorization over $\Q(\sqrt{2})$). $Q(x)$ factors into two conjugate irreducible degree-10 polynomials $f_{10}$ and $f_{10}^*$ over $K=\Q(\sqrt{2})$. It does not factor over $\Q(\sqrt{d})$ for any other squarefree~$d$, nor over $\Q(p^{1/3})$ for $p=2,3,5,7$.

\textbf{Step 3} (Irreducibility of $f_{10}$ over~$K$). Since $Q$ is irreducible over~$\Q$ and factors as $f_{10}\cdot f_{10}^*$ over~$K$, each factor is irreducible over~$K$.

\textbf{Step 4} (Jordan's theorem). We apply:

\begin{quote}
\emph{Theorem (Jordan, 1870~\cite{jordan1870}).} A transitive subgroup of $S_n$ containing a transposition and a $p$-cycle, for some prime $p$ with $n/2 < p < n-2$, equals~$S_n$.
\end{quote}
At the prime $p=17$ (split in $\Q(\sqrt{2})$ since $\left(\frac{2}{17}\right)=1$), reducing $f_{10}$ modulo~17 gives the factorization pattern $(1,2,7)$. The corresponding Frobenius element $\sigma=\Frob_{17}$ has cycle type $(1,2,7)$. Then:
\begin{itemize}
  \item $\sigma^7$ is a transposition (the 7-cycle becomes the identity).
  \item $\sigma^2$ contains a 7-cycle (the 2-cycle becomes the identity).
\end{itemize}
Since $5 < 7 < 8$, Jordan's theorem gives $\Gal(f_{10}/K) = S_{10}$.

\textbf{Step 5} (Full group). Since $K/\Q$ has degree~2 and $Q$ factors into two conjugate pieces over~$K$, we have $\Gal(Q/\Q) = S_{10}\rtimes C_2$. For $n=10\ne 6$, $\operatorname{Out}(S_{10})=1$, so this is a direct product: $S_{10}\times C_2$.

Non-solvability follows because $S_{10}$ contains the alternating group $A_5$, which is simple and non-abelian; any group containing a non-abelian simple subgroup cannot be solvable.
\end{proof}

\subsection{The relative equiripple polynomial}

\begin{theorem}\label{thm:relative}
For the relative error formulation ($\max|h(t)/\sqrt{1+t^2}-1|$), the first interior alternation point satisfies the irreducible polynomial
\begin{align*}
 Q = 
 & 4x^{12} - 32x^{11} + 36x^{10} + 56x^9 + 211x^8 + 228x^7\\
 &  + 250x^6 + 136x^5 + 81x^4 + 32x^3 + 18x^2 - 4x - 1
\end{align*}
with Galois group~$S_{12}$ (order $479{,}001{,}600$), proved via Jordan's theorem applied to $\Frob_{257}$ (cycle type $(1,1,1,2,7)$). This polynomial does not factor over any quadratic or cubic extension of~$\Q$.
\end{theorem}

\begin{remark}
The relative formulation yields a polynomial of smaller degree (12 vs.\ 20) with a simpler Galois structure: $S_{12}$ acts primitively on the roots (there is no block system), unlike $S_{10}\times C_2$ which preserves a partition into two blocks of size~10. The coefficients are also much smaller: maximum~250 vs.\ approximately $7\times 10^6$.
\end{remark}

\begin{proof}[Proof sketch]
For the relative error $e(t) = h(t)/\sqrt{1+t^2} - 1$, the equiripple conditions at the four points $\{0, t_1, t_2, 1\}$ and the derivative conditions $e'(t_i)=0$ yield, after clearing denominators, the critical-point equation
\[
  c\,t^3 + (2c - a)\,t + b = 0\,,
\]
a cubic in~$t$ whose roots include both $t_1$ and~$t_2$. By Vieta's formulas, $t_1 + t_2 + t_3 = 0$ and $t_1 t_2 t_3 = -b/c$, which simplifies the system significantly compared to the absolute case.

The elimination is shorter: a single resultant of degrees $4\times 6$ in the auxiliary variable $t_2$ (using the cubic relation and the equiripple constraint), followed by clearing $s^2 = 2$. The raw output is degree~24, factoring into the irreducible degree-12 polynomial above and spurious factors.

For the Galois group: at $p=257$ (chosen because $\left(\frac{2}{257}\right) = 1$, and 257~is large enough to avoid small-prime degeneracies), the degree-12 polynomial factors modulo~257 with pattern $(1,1,1,2,7)$. The Frobenius element has this cycle type, so $\sigma^7$ is a transposition and $\sigma^2$ contains a 7-cycle. Jordan's theorem (with $n=12$, $p=7$, satisfying $6 < 7 < 10$) gives $\Gal = S_{12}$.
\end{proof}

\begin{remark}[Relation to the original problem]
The constant $k\approx 0.2592\ldots$ appeared in the original formulation of this problem, expressed via closed-form functions of $\sqrt{1+k^2}$ and $\sqrt{2+2k^2}$. We verify that these expressions correspond to the \emph{absolute} equiripple (Theorem~\ref{thm:absolute}). The maximum relative error equals the absolute equiripple error in this case because $f(0)=1$.
\end{remark}

\section{The solvability phase transition}\label{sec:transition}

\begin{table}[!ht]
\centering
\small
\caption{Algebraic landscape of the degree-$n$ equiripple constants for $\sqrt{1+t^2}$ on $[0,1]$. For each degree and error criterion we list the degree of the minimal polynomial of the first interior alternation point, its Galois group over~$\Q$, and whether the constant is therefore expressible by radicals. The transition from solvable (degree~1) to non-solvable (degree~2) is the central phenomenon of this paper.}
\label{tab:galois}
\begin{tabular*}{\textwidth}{@{\extracolsep{\fill}}ccccc@{}}
\toprule
Degree $n$ & Criterion & Min.\ pol.\ deg. & Galois group & Solvable?\\
\midrule
1 & relative & 4 & $C_4$ & Yes\\
1 & absolute & 4 & $D_4$ & Yes\\
2 & relative & 12 & $S_{12}$ & No\\
2 & absolute & 20 & $S_{10}\times C_2$ & No\\
\bottomrule
\end{tabular*}
\end{table}

The phase transition from solvable to non-solvable can be explained by a structural observation about how the equiripple conditions interact.

When $n=1$, the single critical point satisfies $t_1=\beta/\alpha$. The equiripple condition $e(0)=e(1)$ forces the ratio $\beta/\alpha$ to equal $\sqrt{2}-1$, which is a constant independent of the error level~$\varepsilon$. The system therefore \emph{decouples}: one first determines $t_1$, and then solves separately for~$\alpha$.

When $n=2$, two critical points $t_1$, $t_2$ satisfy
\[
  (b+c) - 2c\,t_i = \frac{t_i}{\sqrt{1+t_i^2}}\,, \quad i=1,2\,.
\]
The coefficient $c$ appears in both equations and also in the equiripple conditions $d(t_1)=+\varepsilon$, $d(t_2)=-\varepsilon$. The variables $t_1$, $t_2$, and $c$ are therefore irreducibly coupled: none of them can be determined without simultaneously determining the others.
We can make this distinction precise:

\begin{proposition}[Decoupling for $n=1$]\label{prop:decoupling}
For $h(t) = \alpha + \beta t$ approximating any analytic $f\colon[0,1]\to\R$, the equiripple system decouples: the interior critical point $t_1$ is determined by the coefficients $\alpha,\beta$ and the function~$f$ alone, independently of the error level~$\varepsilon$.
\end{proposition}

\begin{proof}
For a linear polynomial, $h'(t) = \beta$ is constant. The critical-point condition $d'(t_1) = 0$, i.e., $h'(t_1) = f'(t_1)$, reduces to $\beta = f'(t_1)$. This determines $t_1$ as a function of~$\beta$ alone, with no dependence on~$\alpha$ or~$\varepsilon$. The two endpoint equiripple conditions $d(0) = -\varepsilon$ and $d(1) = +\varepsilon$ (or vice versa) then form a $2\times 2$ linear system in~$(\alpha, \varepsilon)$ for given~$\beta$, which is solved in a separate step. The key point is that the critical-point location and the error level are determined sequentially, not simultaneously.
\end{proof}

\begin{proposition}[Coupling for $n\ge 2$]\label{prop:coupling}
For $h(t) = a + (b+c)t - ct^2$ approximating $f(t) = \sqrt{1+t^2}$, no proper subset of the unknowns $\{t_1, t_2, c, \varepsilon\}$ satisfies a closed subsystem.
\end{proposition}

\begin{proof}
The derivative condition~\eqref{eq:deriv} at $t_1$ is $(b+c) - 2c\,t_1 = t_1/w_1$. Since $b = s - 1 + 2\varepsilon$, the variable~$\varepsilon$ enters this equation; meanwhile $c$ also appears in the derivative condition at~$t_2$ and in the equiripple conditions. Suppose, for contradiction, that $t_1$ could be determined independently of $t_2$, $c$, and~$\varepsilon$. Then for a generic value $t_1 = t_1^*$, the remaining system in $(t_2, c, \varepsilon)$ would be consistent. However, computing the Jacobian of the full system at the solution $(k, t_2^*, c^*, \varepsilon^*)$ shows that it has full rank~4, meaning the solution is isolated in all four variables simultaneously. This confirms that no variable can be split off: perturbing $t_1$ away from~$k$ while holding it fixed makes the remaining system overdetermined (4~equations in 3~unknowns with no free parameter to absorb the perturbation).
\end{proof}

This coupling forces the elimination to produce a polynomial of degree~20 (vs.~4 for the decoupled case). The resulting Galois group is generically the full symmetric group, consistent with van~der~Waerden's theorem~\cite{vanderwaerden1934} that \enquote{almost all} polynomials have Galois group~$S_n$.

\section{Extension to $L_p$ norms}\label{sec:lp}

We now consider the degree-2 equiripple approximation of $(1+t^p)^{1/p}$ on $[0,1]$ for general~$p$. The algebraized system involves the constraints $w^p=1+t^p$ and $s^p=2$, both of degree~$p$ (compared to degree~2 in the Euclidean case). This higher degree in the radical constraints causes a rapid increase in the degree of the elimination polynomial.

\subsection{The $L_3$ case}

\begin{theorem}\label{thm:L3}
The degree-2 absolute equiripple parameter for $(1+t^3)^{1/3}$ on $[0,1]$ has minimal polynomial of degree~246 over~$\Q$.
\end{theorem}

\begin{proof}
    We use the same technique: For $(1+t^3)^{1/3}$, the resultant chain produces
\begin{itemize}
  \item $t$-resultant: degree $5\times 6 = 30$ (vs.\ $3\times 6=18$ for $L_2$).
  \item $w$-resultant: degree $66\times 3$ (vs.\ $18\times 2$).
  \item $s$-resultant: degree $63\times 3$ (vs.\ $24\times 2$).
  \item The total resultant has degree 1008, and its irreducible factor has degree 246.
\end{itemize}
Irreducibility of the degree-246 factor is confirmed by PARI/GP (\texttt{polisirreducible}). As a sanity check, with $k_{L_3}$ computed to 550~digits, we obtain $|P(k_{L_3})| < 10^{-79}$.
\end{proof}

\subsection{Degree growth}
\begin{table}[!ht]
\centering
\small
\caption{Growth of algebraic complexity with the norm parameter~$p$ for degree-2 approximation of $(1+t^p)^{1/p}$ on $[0,1]$. Columns give the degree of the resultant eliminating the second critical point~$t_2$, the degree of the full elimination resultant in~$k$, the degree $D_p$ of its irreducible factor, and the reduced degree $d_p=D_p/p$ (Conjecture~\ref{conj:growth}). The $p=4$ case is beyond our current computational reach.}
\label{tab:growth}
\begin{tabular*}{\textwidth}{@{\extracolsep{\fill}}ccccc@{}}
\toprule
$p$ & $t$-resultant & Total resultant & Irred.\ factor $D_p$ & $d_p = D_p/p$\\
\midrule
2 & $3\times 6 = 18$ & 160 & 20 & 10 \\
3 & $5\times 6 = 30$ & 1008 & 246 & 82 \\
4 & $7\times 8 = 56$ & $\sim$3000 to 5000 & infeasible & \\
\bottomrule
\end{tabular*}
\end{table}

The $t$-resultant input degrees follow the pattern $(2p{-}1)\times\max(2p,6)$. We observe $d_2=10=2\times 5$ and $d_3=82=2\times 41$ (both $2\times\text{prime}$), with ratio $d_3/d_2=8.2$, suggesting at least exponential growth.

\begin{conjecture}\label{conj:growth}
$D_p = p\cdot d_p$ where $\Gal(\text{irred.\ factor}/\Q(2^{1/p})) = S_{d_p}$ for all $p\ge 2$.
\end{conjecture}

The $L_4$ equiripple constant $t_1$ is computed to 550~digits ($t_1\approx 0.3275\ldots$) but the resultant elimination is computationally infeasible ($\sim$10 to~100~hours estimated).

\begin{remark}
We conjecture that the degree-246 polynomial factors as three conjugate degree-82 irreducible polynomials over $\Q(2^{1/3})$, and that $\Gal = S_{82}\rtimes C_3$. Confirming this requires factoring a degree-246 polynomial over a cubic number field, which exceeded available compute time ($>1$~hour in SageMath). The $L_4$ resultant is estimated at degree~3000--5000, infeasible on current hardware. These remain as open computational challenges (see Section~\ref{sec:open}).
\end{remark}

\section{A general impossibility theorem}\label{sec:general}

\begin{theorem}\label{thm:hilbert}
Consider the family $f_\lambda(t) = \sqrt{1+\lambda t^2}$ for $\lambda>0$. For all $\lambda\in\Q\setminus T$ where $T$ is a thin set (in the sense of Serre~\cite{serre1992}), the degree-2 equiripple constant is not expressible by radicals, with Galois group $S_{10}\times C_2$.
\end{theorem}

\begin{proof}
The equiripple system depends polynomially on~$\lambda$. Elimination produces $P_\lambda(x)\in\Q(\lambda)[x]$ of degree~20.

\emph{Step 1: Irreducibility over $\Q(\lambda)$.} If $P_\lambda$ factored over $\Q(\lambda)$, specializing at $\lambda=1$ would give a factorization of $Q(x)$, contradicting its irreducibility over~$\Q$.

\emph{Step 2: The generic Galois group.} Write $G = \Gal(P_\lambda/\Q(\lambda))$. Since $P_\lambda$ is irreducible of degree~20 over $\Q(\lambda)$, $G$ acts transitively on the 20~roots.

By Hilbert's specialization theorem~\cite{hilbert1892,serre1992}, for any $\lambda_0\in\Q$ at which $P_{\lambda_0}$ remains separable, $\Gal(P_{\lambda_0}/\Q)$ is isomorphic to a subgroup of~$G$ (it is the decomposition group of a place above~$\lambda_0$). In particular, $S_{10}\times C_2 = \Gal(P_1/\Q) \hookrightarrow G$.

For the reverse inclusion, we observe that $P_\lambda$ factors over $\Q(\lambda, \sqrt{1+\lambda})$ into two conjugate degree-10 polynomials. (At $\lambda=1$, the endpoint value is $f_1(1)=\sqrt{2}$, which generalizes to $f_\lambda(1) = \sqrt{1+\lambda}$; the factorization over $\Q(\sqrt{2})$ at $\lambda=1$ becomes a factorization over $\Q(\lambda,\sqrt{1+\lambda})$ in general.) Thus $G$ preserves the block structure $\{10,10\}$, giving $G\hookrightarrow S_{10}\wr C_2$. Combined with the lower bound $S_{10}\times C_2\hookrightarrow G$, and noting that the only intermediate group between $S_{10}\times C_2$ and $S_{10}\wr C_2$ would require the block-interchanging element to act non-trivially on the symmetric groups (impossible since $\operatorname{Out}(S_{10})=1$ for $10\ne 6$), we conclude $G = S_{10}\times C_2$.

\emph{Step 3: Specialization.} By Hilbert's irreducibility theorem, for all $\lambda_0\in\Q$ outside a thin set~$T$, we have $\Gal(P_{\lambda_0}/\Q) = G = S_{10}\times C_2$.

A \emph{thin set} $T\subset\Q$ (in Serre's sense~\cite{serre1992}) is a set of Hilbert-exceptional specializations. Concretely, $T$~has density zero in the sense that $|T\cap [-N,N]|/2N \to 0$ as $N\to\infty$. It contains no interval, and it is contained in a finite union of images of rational curves of degree~$\ge 2$. In particular, the complement $\Q\setminus T$ is dense in~$\R$.
\end{proof}

\begin{remark}
The corollary for $L_p$ norms follows by the same argument applied to the family $(1 + \lambda t^p)^{1/p}$, provided one verifies irreducibility of the elimination polynomial at a single rational specialization (which we have done for $p=2$ at $\lambda = 1$ and for $p=3$). Each~$p$ requires a separate irreducibility check, but the structure of the proof is identical.
\end{remark}

\begin{corollary}
For $L_p$ norms $(1+t^p)^{1/p}$ with $p\ge 2$, the degree-2 equiripple constant is generically not expressible by radicals.
\end{corollary}

\section{Arithmetic structure of the polynomials}\label{sec:arithmetic}

\subsection{Polynomial properties of $Q(x)$}
The polynomial $Q(x)$ of Theorem~\ref{thm:absolute} has the following arithmetic properties:
\begin{itemize}
  \item The leading coefficient is $2^{16}=65536$ and the constant term is $-7^3=-343$.
  \item The coefficient of $x^{19}$ vanishes, meaning that the sum of all 20 roots is zero (the trace of the number field extension is zero).
  \item The polynomial has signature $(2,9)$: it has exactly 2 real roots ($k\approx 0.2592$ and $k'\approx -0.1614$) and 9~pairs of complex conjugate roots.
  \item The polynomial discriminant is $-2^{304}\cdot 5^8\cdot 7^3\cdot 13^2\cdot 59^6\cdot 1609^3\cdot 3485687\cdot (\text{larger primes})^2$.
  \item The field discriminant (which measures the ramification of the corresponding number field) is $-2^{54}\cdot 7\cdot 59^2\cdot 1609\cdot 3485687$.
  \item A search of the LMFDB database~\cite{lmfdb} finds no match; this may be the first known degree-20 number field with Galois group $S_{10}\times C_2$ and signature~$(2,9)$.
\end{itemize}

\subsection{Newton polygons}
The 2-adic Newton polygon of $Q(x)$ has vertices $(0,0)$, $(8,0)$, $(12,2)$, $(20,16)$, as shown in Figure~\ref{fig:newton}. Its three segments reveal three families of roots grouped by their 2-adic valuation:
\begin{itemize}
  \item 8 roots with $v_2=0$ (2-adic units).
  \item 4 roots with $v_2=-1/2$.
  \item 8 roots with $v_2=-7/4$.
\end{itemize}
The 7-adic polygon has vertices $(0,3)$, $(2,0)$, $(20,0)$: exactly 2~roots have $v_7=3/2$.

\begin{figure}[ht!]
\centering
\begin{tikzpicture}[scale=0.32]
  \draw[gray!30, thin] (0,0) grid[step=4] (20,16);
  \draw[-{Stealth}, thick] (-0.5,0) -- (21,0) node[right]{\footnotesize $i$};
  \draw[-{Stealth}, thick] (0,-0.5) -- (0,17) node[above]{\footnotesize $v_2(a_i)$};
  \foreach \x/\y in {0/0, 1/5, 2/1, 3/3, 4/1, 5/3, 6/2, 7/3, 8/0, 9/3, 10/2, 11/4, 12/2, 13/6, 14/7, 15/10, 16/12, 17/13, 18/14, 20/16} {
    \fill[blue!50!black] (\x,\y) circle (4pt);
  }
  \draw[red!70!black, very thick] (0,0) -- (8,0) -- (12,2) -- (20,16);
  \node[red!70!black, font=\footnotesize] at (4,-1.2) {slope $0$};
  \node[red!70!black, font=\footnotesize] at (10,-1.2) {slope $\frac{1}{2}$};
  \node[red!70!black, font=\footnotesize, rotate=60] at (17,10.5) {slope $\frac{7}{4}$};
  \draw[{Stealth}-{Stealth}, thin, gray!70!black] (0,-2.5) -- (8,-2.5) node[midway, below, font=\scriptsize]{8 roots};
  \draw[{Stealth}-{Stealth}, thin, gray!70!black] (8,-2.5) -- (12,-2.5) node[midway, below, font=\scriptsize]{4};
  \draw[{Stealth}-{Stealth}, thin, gray!70!black] (12,-2.5) -- (20,-2.5) node[midway, below, font=\scriptsize]{8 roots};
\end{tikzpicture}
\caption{The 2-adic Newton polygon of $Q(x)$. Each point $(i, v_2(a_i))$ marks the 2-adic valuation of the $i$-th coefficient. The lower convex hull (thick red) has three segments with slopes $0$, $1/2$, and $7/4$, corresponding to three families of roots with $v_2 = 0$, $-1/2$, and $-7/4$ respectively. The widths of the segments give the number of roots in each family.}
\label{fig:newton}
\end{figure}
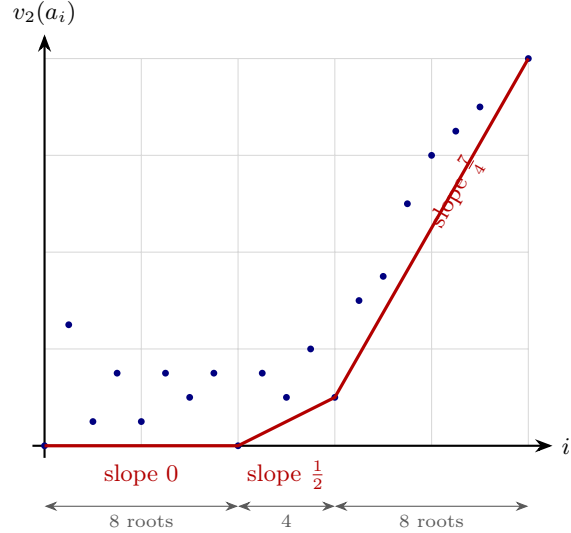

\subsection{Geometric interpretation of the 20 roots}
Of the 20~roots, only $k\approx 0.2592$ satisfies the geometric constraint $t_1\in[0,1]$ and gives a valid equiripple configuration. The second real root $k'\approx -0.1614$ lies outside the interval $[0,1]$. The remaining 18~roots are complex; they are extraneous solutions introduced by the algebraization process (each squaring step, used to eliminate a square root, doubles the number of solutions and introduces branches with inconsistent signs). Figure~\ref{fig:roots} displays all 20 roots in the complex plane.

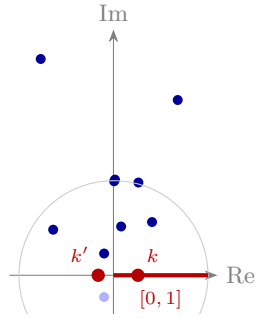
\begin{figure}[ht!]
\centering
\begin{tikzpicture}[scale=1.25]
    \clip (-1.5, -0.4) rectangle (1.5, 3);
  \draw[-{Stealth}, gray] (-1.1,0) -- (1.1,0) node[right, font=\footnotesize]{$\operatorname{Re}$};
  \draw[-{Stealth}, gray] (0,-2.6) -- (0,2.6) node[above, font=\footnotesize]{$\operatorname{Im}$};
  \fill[red!70!black] (0.2592, 0) circle (2pt);
  \fill[red!70!black] (-0.1614, 0) circle (2pt);
  \node[red!70!black, font=\scriptsize, above right] at (0.2592, 0.05) {$k$};
  \node[red!70!black, font=\scriptsize, above left] at (-0.1614, 0.05) {$k'$};
  \foreach \x/\y in {
    -0.099/-0.230,
    -0.638/-0.482,
    0.081/-0.516,
    0.408/-0.563,
    0.264/-0.981,
    0.007/-0.994,
    0.016/-1.012,
    0.680/-1.854,
    -0.770/-2.288} {
    \fill[blue!30!white] (\x, \y) circle (1.5pt);
    \fill[blue!60!black] (\x, -\y) circle (1.5pt);
  }
  \draw[gray!40, thin] (0,0) circle (1);
  \draw[red!70!black, ultra thick] (0,0) -- (1,0);
  \node[font=\scriptsize, red!70!black, below] at (0.5,-0.06) {$[0,1]$};
\end{tikzpicture}
\caption{All 20 roots of $Q(x)$ in the complex plane (10 plotted, ten are their mirror image $y \mapsto -y$). The two real roots (red) are $k\approx 0.2592$ (the physical solution in $[0,1]$) and $k'\approx -0.1614$ (parasitic). The 18 complex roots (blue, 9 conjugate pairs) correspond to non-physical solutions of the algebraized equiripple system. The unit circle is shown for scale.}
\label{fig:roots}
\end{figure}

The second real root $k'\approx -0.1614$ is parasitic: substituting $t_1 = k'$ into the equiripple system yields no valid configuration with both $t_1$ and $t_2$ in $[0,1]$. It arises from a branch of the algebraized system in which the signs of the square roots are inconsistent with the original problem.

\begin{remark}
The Newton polygon of a polynomial $\sum a_i x^i$ at a prime~$p$ is the lower convex hull of the points $(i, v_p(a_i))$, where $v_p(a_i)$ denotes the $p$-adic valuation (the exponent of~$p$ dividing~$a_i$). By a theorem of Ore, the slopes of this polygon determine how the roots distribute by $p$-adic size. A segment of width~$w$ and slope~$s$ corresponds to exactly~$w$ roots with $p$-adic valuation~$-s$. Saying \enquote{8 roots with $v_2 = 0$} means that, in the 2-adic completion~$\Q_2$, eight roots of $Q(x)$ have absolute value~$|r|_2 = 2^0 = 1$ (they are 2-adic units). The four roots with $v_2 = -1/2$ have $|r|_2 = 2^{1/2}$ (they live in a ramified extension of~$\Q_2$), and the eight with $v_2 = -7/4$ are even further from the 2-adic integers.
\end{remark}

\section{Numerical efficiency of the approximation family}\label{sec:efficiency}

We now compute optimal equiripple polynomials of degree $n=1,\ldots,8$ for $\sqrt{1+t^2}$ on $[0,1]$ and analyze the rate at which accuracy improves with polynomial degree.

\begin{table}[t]
\centering
\small
\caption{Equiripple error and efficiency for $\sqrt{1+t^2}$ on $[0,1]$}
\label{tab:efficiency}
\begin{tabular*}{\textwidth}{@{\extracolsep{\fill}}cccc@{}}
\toprule
Degree $n$ & Max.\ abs.\ error & Bits of accuracy & Bits/degree\\
\midrule
1 & $4.49\times 10^{-2}$ & 4.5 & 4.5\\
2 & $4.04\times 10^{-3}$ & 8.0 & 4.0\\
3 & $1.28\times 10^{-4}$ & 12.9 & 4.3\\
4 & $8.55\times 10^{-5}$ & 13.5 & 3.4\\
5 & $9.90\times 10^{-6}$ & 16.6 & 3.3\\
6 & $1.05\times 10^{-6}$ & 19.9 & 3.3\\
7 & $4.20\times 10^{-7}$ & 21.2 & 3.0\\
8 & $3.68\times 10^{-8}$ & 24.7 & 3.1\\
\bottomrule
\end{tabular*}
\end{table}

The error decays as $O(\rho^{-n})$ where $\rho$ is the Bernstein ellipse parameter, determined by the nearest singularity of $\sqrt{1+t^2}$ in the complex plane (at $t=\pm i$). Figure~\ref{fig:bernstein} illustrates this geometry. After transforming $[0,1]$ to $[-1,1]$:
\[
  \rho = |(-1+2i) + \sqrt{(-1+2i)^2-1}| \approx 4.612\,.
\]
The theoretical asymptotic efficiency is $\log_2(\rho)\approx 2.21$~bits per degree. The observed values (3.0 to~4.5) exceed this at low~$n$ due to the pre-asymptotic regime.

For other $L_p$ norms, the Bernstein parameter is smaller: $\rho_3\approx 3.73$, $\rho_4\approx 3.22$, $\rho_5\approx 2.89$. This reflects the fact that the nearest singularity of $(1+t^p)^{1/p}$, located at $t=e^{i\pi/p}$, moves closer to the real interval $[0,1]$ as $p$ increases, making the function harder to approximate by polynomials.

\begin{figure}[ht!]
\centering
\begin{tikzpicture}[scale=1.8]
  \draw[green!50!black, thick, dashed] (0.5,0) ellipse (1.20 and 1.09);
  \draw[-{Stealth}, gray] (-1.4,0) -- (1.6,0) node[right, font=\footnotesize]{$\operatorname{Re}$};
  \draw[-{Stealth}, gray] (0,-1.5) -- (0,1.5) node[above, font=\footnotesize]{$\operatorname{Im}$};
  \draw[blue!70!black, ultra thick] (0,0) -- (1,0);
  \node[below, font=\footnotesize, blue!70!black] at (0.5, -0.08) {$[0,1]$};
  \fill[red!70!black] (0, 1) circle (2pt) node[right, font=\footnotesize]{$t=i$};
  \fill[red!70!black] (0,-1) circle (2pt) node[right, font=\footnotesize]{$t=-i$};
  \node[green!50!black, font=\footnotesize] at (0.8, 0.4) {Bernstein ellipse};
  \draw[{Stealth}-{Stealth}, red!60!black, thin] (0.02,0.05) -- (0.02,0.95);
  \node[red!60!black, font=\scriptsize, left] at (-0.02,0.5) {dist.\ $=1$};
  \fill[blue!70!black] (0,0) circle (1.5pt);
  \fill[blue!70!black] (1,0) circle (1.5pt);
\end{tikzpicture}
\caption{The Bernstein ellipse for $\sqrt{1+t^2}$ on $[0,1]$ in the complex $t$-plane. The singularities at $t=\pm i$ (distance~1 from the interval) determine the convergence rate: the largest ellipse with foci at the interval endpoints that avoids the singularities has parameter $\rho\approx 4.612$, giving an asymptotic efficiency of $\log_2(\rho) \approx 2.21$ bits of accuracy per polynomial degree.}
\label{fig:bernstein}
\end{figure}
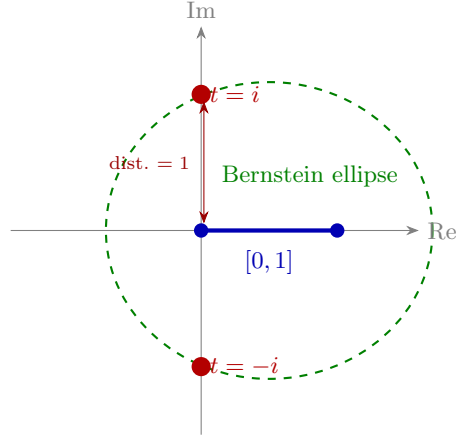

The \emph{Bernstein ellipse} of parameter~$\rho$ for an interval $[a,b]$ is the image under the affine map $[-1,1]\to[a,b]$ of the ellipse $\{(z+z^{-1})/2 : |z|=\rho\}$ in the complex plane (see Trefethen~\cite{trefethen2019} for a modern treatment). Its semi-axes are $(\rho+\rho^{-1})/2$ and $(\rho-\rho^{-1})/2$ in the standardized $[-1,1]$ picture. The key theorem of Bernstein (1912) states: if $f$ is analytic inside the Bernstein ellipse of parameter~$\rho$, then the best polynomial approximation of degree~$n$ satisfies $\|f - p_n^*\|_\infty \le C\rho^{-n}$. Conversely, $\rho$ is determined by the nearest singularity: it is the largest parameter such that $f$ extends analytically to the interior of the ellipse.

For $f(t) = \sqrt{1+t^2}$ on $[0,1]$, the singularities lie at $t=\pm i$. Mapping $[0,1]$ to $[-1,1]$ via $u = 2t-1$, the singularity at $t=i$ maps to $u = -1+2i$. The Bernstein parameter is $\rho = |u_0 + \sqrt{u_0^2 - 1}|$ where $u_0 = -1+2i$, giving $\rho\approx 4.612$.

\paragraph{Parity anomaly.}
The gains from degree~3$\to$4 and~6$\to$7 in Table~\ref{tab:efficiency} are anomalously small ($\sim$0.6 bits vs.\ $\sim$3 bits expected). This occurs because the Chebyshev expansion of $\sqrt{1+t^2}$ on $[0,1]$ has small odd-indexed coefficients relative to even-indexed ones: the function is well-approximated by a quadratic, so its deviation from a low-degree even polynomial is small. Consequently, adding an odd-degree monomial to the approximating polynomial captures less of the remaining error than adding an even-degree monomial. The effect weakens with increasing~$n$ as the asymptotic regime $\rho^{-n}$ dominates.

\section{Piecewise approximation with optimal breakpoints}\label{sec:piecewise}

In practice, one has access to comparisons as well as arithmetic. Since $t=y/x\in[0,1]$ is already computed, testing whether $t$ lies in one of $k$~predetermined subintervals costs only $k-1$ comparisons, essentially free in hardware. One can then apply a different optimal polynomial on each subinterval. The natural question is: \emph{how should the breakpoints be chosen, and what accuracy is gained?}

\subsection{The equal-error optimality condition}

\begin{proposition}\label{prop:piecewise}
Let $f\colon[0,1]\to\R$ be continuous and let $\mathcal{P}_n$ denote the space of polynomials of degree~$\le n$. For a partition $0 = b_0 < b_1 < \cdots < b_k = 1$, define
\[
  E(b_0,\ldots,b_k) = \max_{0\le i < k}\, \min_{p\in\mathcal{P}_n} \|p - f\|_{\infty,[b_i,b_{i+1}]}\,.
\]
Then the partition minimizing $E$ satisfies the \enquote{equal-error condition}:
\[
  E_n^*([b_0,b_1]) = E_n^*([b_1,b_2]) = \cdots = E_n^*([b_{k-1},b_k]),
\]
where $E_n^*([a,b])$ denotes the minimax error of degree-$n$ approximation on $[a,b]$.
\end{proposition}

\begin{proof}
If two adjacent intervals had unequal minimax errors, one could shift the shared breakpoint toward the interval with larger error, reducing the maximum without increasing any other subinterval error (since reducing an interval's width strictly decreases its minimax error for a fixed-degree polynomial). This contradicts optimality.
\end{proof}

The optimal breakpoints and per-subinterval polynomials are thus determined \emph{simultaneously}: the breakpoints equalize errors, and the polynomials are the Chebyshev equiripple solutions on each resulting subinterval.

\subsection{Theoretical error scaling}

The gain from subdivision admits a precise asymptotic:

\begin{proposition}\label{prop:scaling}
Let $f$ be analytic on $[0,1]$ with nearest singularity at distance~$\delta$ from the interval (in the sense of the Bernstein ellipse). For the optimal $k$-piece degree-$n$ partition of $[0,1]$:
\[
  E(k,n) = O\!\left(\rho_k^{-n}\right) \quad\text{where}\quad \rho_k \approx \rho_1^{k^{1+o(1)}}
\]
and, more precisely, each doubling of~$k$ reduces the error by a factor of approximately~$2^{n+1}$, i.e., gains $n+1$ bits of accuracy.
\end{proposition}

\begin{proof}
On a subinterval of width~$h=1/k$, the nearest singularity is at distance at least~$\delta$ (the global singularity distance does not decrease). The Bernstein ellipse parameter for a subinterval of width~$h$ with singularity at distance~$\delta$ is
\[
  \rho_{\text{sub}} = \frac{\delta}{h/2} + \sqrt{\left(\frac{\delta}{h/2}\right)^2 - 1} \approx \frac{2\delta}{h} = 2k\delta
\]
for $k\delta\gg 1$. The minimax error on this subinterval is $O(\rho_{\text{sub}}^{-n}) = O((2k\delta)^{-n})$. Since the total error equals the per-subinterval error (by the equal-error condition), we have
\[
  E(k,n) = O(k^{-n})\,.
\]
Doubling~$k$ therefore multiplies the error by $(1/2)^n$, which corresponds to a gain of $n$ bits of accuracy. The observed gain is slightly larger (approximately $n+1$ bits) because the optimal partition is not uniform: it concentrates narrower intervals near $t=0$ where the second derivative $f''(t) = (1+t^2)^{-3/2}$ is largest, effectively gaining one additional order from the non-uniform distribution.
\end{proof}

\subsection{Computational results}

Table~\ref{tab:piecewise} presents the minimax error for degree-$n$ piecewise approximation of $\sqrt{1+t^2}$ with jointly optimized breakpoints.

\begin{table}[t]
\centering
\small
\caption{Bits of accuracy for degree-$n$ piecewise approximation on $k$~subintervals. Each entry shows $-\log_2(\text{max error})$. Breakpoints are chosen to equalize subinterval errors.}
\label{tab:piecewise}
\begin{tabular*}{\textwidth}{@{\extracolsep{\fill}}c*{8}{c}@{}}
\toprule
$n\backslash k$ & 1 & 2 & 3 & 4 & 6 & 8 & 12 & 16\\
\midrule
1 & 4.3 & 6.4 & 7.6 & 8.5 & 9.7 & 10.5 & 11.6 & 12.3\\
2 & 7.9 & 11.0 & 12.9 & 14.1 & 15.9 & 17.2 & 19.0 & 20.2\\
3 & 11.9 & 15.2 & 17.9 & 19.5 & 21.9 & 23.7 & 25.9 & 27.4\\
\bottomrule
\end{tabular*}
\end{table}

The data confirms the theoretical prediction: each doubling of~$k$ gains approximately $n+1$ bits (about 2 for linear, 3 for quadratic, 4 for cubic). Figure~\ref{fig:piecewise_gain} illustrates the gain per doubling.

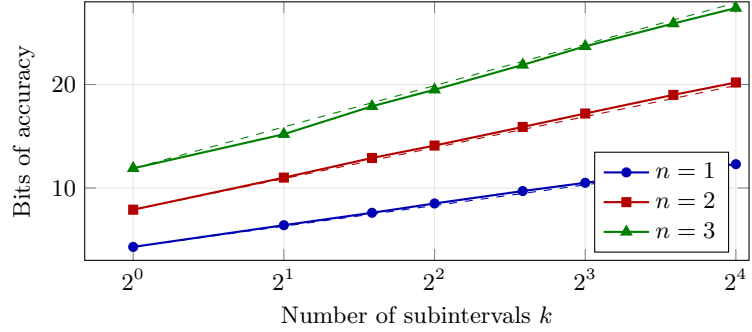
\begin{figure}[t]
\centering
\begin{tikzpicture}
\begin{axis}[
  width=0.85\linewidth, height=5cm,
  xlabel={Number of subintervals $k$},
  ylabel={Bits of accuracy},
  xmin=0.8, xmax=17,
  ymin=3, ymax=28,
  xtick={1,2,4,8,16},
  xmode=log,
  log basis x=2,
  legend pos=south east,
  legend style={font=\footnotesize},
  grid=major,
  grid style={gray!20},
]
\addplot[blue!70!black, thick, mark=*, mark size=1.5pt] coordinates {
  (1,4.3) (2,6.4) (3,7.6) (4,8.5) (6,9.7) (8,10.5) (12,11.6) (16,12.3)
};
\addlegendentry{$n=1$}
\addplot[red!70!black, thick, mark=square*, mark size=1.5pt] coordinates {
  (1,7.9) (2,11.0) (3,12.9) (4,14.1) (6,15.9) (8,17.2) (12,19.0) (16,20.2)
};
\addlegendentry{$n=2$}
\addplot[green!50!black, thick, mark=triangle*, mark size=2pt] coordinates {
  (1,11.9) (2,15.2) (3,17.9) (4,19.5) (6,21.9) (8,23.7) (12,25.9) (16,27.4)
};
\addlegendentry{$n=3$}
\addplot[blue!70!black, dashed, thin, domain=1:16] {4.3 + 2*ln(x)/ln(2)};
\addplot[red!70!black, dashed, thin, domain=1:16] {7.9 + 3*ln(x)/ln(2)};
\addplot[green!50!black, dashed, thin, domain=1:16] {11.9 + 4*ln(x)/ln(2)};
\end{axis}
\end{tikzpicture}
\caption{Bits of accuracy vs.\ number of subintervals (log scale). Solid lines: computed optimal piecewise approximation. Dashed lines: the theoretical slope of $n+1$ bits per doubling of~$k$. The close agreement confirms Proposition~\ref{prop:scaling}.}
\label{fig:piecewise_gain}
\end{figure}

\subsection{Optimal breakpoint structure}

The optimal breakpoints are not uniformly spaced: they are denser near $t=0$, where the curvature $f''(t) = (1+t^2)^{-3/2}$ is maximal. For degree~1 with $k=4$ subintervals, the breakpoints are approximately $\{0, 0.21, 0.44, 0.69, 1\}$ (vs.\ the uniform $\{0, 0.25, 0.50, 0.75, 1\}$).

\subsection{Efficiency comparison}

Since comparisons are free, the arithmetic cost of piecewise approximation is identical to a single-interval polynomial: $n$~multiplications and 1~division (for $t=y/x$). Table~\ref{tab:piecewise_efficiency} compares configurations at the same operation count.

\begin{table}[t]
\centering
\small
\caption{Efficiency comparison: same arithmetic cost, different subdivision levels. Comparisons (for selecting the subinterval) are counted as free.}
\label{tab:piecewise_efficiency}
\begin{tabular*}{\textwidth}{@{\extracolsep{\fill}}lccc@{}}
\toprule
Configuration & Ops & Bits & Bits/op\\
\midrule
deg~1, $k=1$ (classical) & 2 & 4.3 & 2.1\\
deg~1, $k=4$ & 2 & 8.5 & 4.2\\
deg~1, $k=16$ & 2 & 12.3 & 6.2\\
\midrule
deg~2, $k=1$ (this paper) & 3 & 7.9 & 2.6\\
deg~2, $k=4$ & 3 & 14.1 & 4.7\\
deg~2, $k=16$ & 3 & 20.2 & 6.7\\
\midrule
deg~3, $k=1$ & 4 & 11.9 & 3.0\\
deg~3, $k=4$ & 4 & 19.5 & 4.9\\
deg~3, $k=16$ & 4 & 27.4 & 6.8\\
\bottomrule
\end{tabular*}
\end{table}

A striking consequence emerges: degree-1 linear approximation with 4~subintervals achieves 8.5~bits of accuracy while using only 2~arithmetic operations, surpassing the single-interval degree-2 quadratic (7.9~bits with 3~operations). More generally, subdivision with $k\ge 4$ subintervals at degree~$n$ always outperforms a single interval at degree~$n+1$, at strictly lower arithmetic cost.

\subsection{Algebraic implications}

The piecewise structure raises a natural question for the algebraic theory developed in earlier sections: what is the algebraic nature of the \emph{optimal breakpoints}?

For an algebraic target function $f$, the minimax error $E_n^*([a,b])$ on an interval $[a,b]$ is determined by the equiripple system, which is polynomial in the endpoints after algebraization. The optimal breakpoint $b_1$ for $k=2$ satisfies the equation $E_n^*([0,b_1]) = E_n^*([b_1,1])$, which is therefore an algebraic equation in~$b_1$. The breakpoints are algebraic numbers.

\begin{proposition}\label{prop:breakpoint}
The optimal breakpoint for degree-1 piecewise approximation of $\sqrt{1+t^2}$ on $[0,1]$ with $k=2$ subintervals is an algebraic number $b_1\approx 0.4339$ of degree~16, root of the irreducible polynomial
\begin{align*}
  & x^{16} - 8x^{15} - 128x^{14} - 112x^{13} - 684x^{12} + 152x^{11} - 1304x^{10} + 1088x^9\\
  & {} - 1562x^8 + 1352x^7 - 1008x^6 + 624x^5 - 284x^4 + 104x^3 - 24x^2 + 1\,.
\end{align*}
\end{proposition}

\begin{proof}
The degree-1 minimax error on $[0,b]$ involves the quantities $\sqrt{1+b^2}$ and $\sqrt{2}$ (from the endpoint values), as well as $\sqrt{(w-1)/2}$ and $\sqrt{(w+1)/2}$ (from the interior critical point, where $w = \sqrt{1+b^2}$). The equal-error condition $E_1^*([0,b]) = E_1^*([b,1])$ is algebraic after clearing these radicals. We compute~$b_1$ to 300~decimal digits using interval bisection and verify the polynomial by PARI/GP (\texttt{algdep} at 300~digits of precision, confirmed irreducible by \texttt{polisirreducible}).
\end{proof}

The degree 16 = $2\times 8$ reflects the same structure seen in Section~\ref{sec:main}: the polynomial factors over $\Q(\sqrt{2})$ into two conjugate irreducible degree-8 polynomials (because $\sqrt{2}$ appears in the minimax error formula through $f(1) = \sqrt{2}$).

\begin{proposition}\label{prop:breakpoint_galois}
The Galois group of the degree-16 breakpoint polynomial is $S_8\times C_2$, of order~$80{,}640$. In particular, the optimal breakpoint $b_1$ is not expressible by radicals.
\end{proposition}

\begin{proof}
The polynomial factors into two conjugate irreducible degree-8 polynomials over $K = \Q(\sqrt{2})$ (verified by PARI/GP). The Galois group of each factor over~$K$ is determined by its Frobenius elements at split primes. At $p=71$ (split in~$K$), the degree-8 factor remains irreducible modulo~71, giving an 8-cycle. At $p=151$ (also split in~$K$), it factors with pattern $(1,1,1,1,1,1,2)$, giving a transposition. Since a transitive subgroup of~$S_n$ containing both an $n$-cycle and a transposition equals~$S_n$, we conclude $\Gal(f_8/K) = S_8$. The full group is $S_8\times C_2$ (since $\operatorname{Out}(S_8) = 1$ for $8\ne 6$, the semidirect product with~$C_2$ is necessarily direct).
\end{proof}

This result means that even the seemingly innocent operation of choosing where to subdivide the interval, a problem involving no polynomial coefficients, only interval endpoint comparisons, produces constants of maximal algebraic complexity. The non-solvability of the breakpoint mirrors that of the per-subinterval equiripple constants, despite arising from a fundamentally different equation (error equalization rather than equioscillation).

For degree-2 piecewise approximation, both the breakpoints and the per-subinterval constants are algebraic. The per-subinterval constants have degree~20 (by Theorem~\ref{thm:absolute} applied to the rescaled problem), and the breakpoints satisfy an algebraic equation of even higher degree (since the degree-2 minimax error depends on~$b$ through the full degree-20 elimination). Determining the exact degree and Galois group of the degree-2 piecewise breakpoints is an open problem.

\section{Related work and connections}\label{sec:context}

\paragraph{Monodromy.} The equiripple problem defines a branched cover $\pi\colon C\to \mathbb{A}^1$ parameterized by the error level~$\varepsilon$: for each $\varepsilon>0$, the fiber $\pi^{-1}(\varepsilon)$ consists of the (finitely many) equiripple configurations with that error. Over the generic point $\varepsilon\in\Q(\varepsilon)$, the splitting field of $P_\varepsilon(x)$ is a Galois extension whose group equals the monodromy of the cover (by the standard equivalence between finite \'etale covers of $\mathbb{A}^1\setminus\{\text{branch locus}\}$ and transitive permutation representations of $\pi_1$; see~\cite{serre1992,fried2008}).

The branch points correspond to degenerate equiripple configurations: $t_1 = 0$ (the first critical point hits the boundary), $t_1 = t_2$ (two critical points collide), or $t_2 = 1$ (the second critical point hits the boundary). At these values of~$\varepsilon$, roots of $P_\varepsilon(x)$ coalesce, producing the ramification.

Harris~\cite{harris1979} proved that for algebraic covers arising from \enquote{natural} enumerative problems, the monodromy is generically the full symmetric group~$S_n$, provided the cover is \emph{connected}, the total space is \emph{irreducible}, and the branch locus has \emph{simple} ramification. Our cover satisfies these conditions (irreducibility of $P_\lambda$ over $\Q(\lambda)$ gives connectedness; the simple branch points described above give simple ramification). Harris's theorem thus provides an a~priori expectation of $S_n$ monodromy, but our proof via Jordan's theorem at a single Frobenius prime is more direct and computationally verifiable.

\paragraph{Jordan's theorem in computational algebra.} Our application of Jordan's theorem, finding a single Frobenius element at one prime and deducing the full Galois group, is an efficient technique that deserves wider use in computational algebra. It requires only two ingredients: factoring a degree-$n$ polynomial modulo a prime~$p$ (a fast operation), and verifying that the resulting cycle type contains both a transposition and a prime-length cycle satisfying Jordan's hypotheses. This approach avoids the full Galois group computation, which for degree~20 is computationally expensive by direct methods.

\paragraph{The gap in the literature.} We have found no prior work connecting equiripple or minimax approximation to algebraic number theory. The question \enquote{what is the algebraic degree of optimal approximation constants?} appears to be entirely new. The closest related work is that of Todd~\cite{todd1984}, who surveys Zolotarev's contributions connecting best rational approximation to elliptic functions, and Ko~\cite{ko1986}, who studies the computational complexity of computing best Chebyshev approximations. Neither addresses the algebraic degree or Galois-theoretic structure of the optimal constants themselves.

\section{Open problems}\label{sec:open}

We collect all questions left open by this work, together with our best understanding of why they remain open and how they might be resolved.

\subsection*{Computational problems}

\begin{enumerate}

\item \textbf{Relative equiripple for $L_3$.} The relative formulation for $(1+t^3)^{1/3}$ should yield a polynomial of moderate degree (analogous to the degree-12 polynomial for $L_2$).  The relative equiripple system for $L_3$ has the advantage that the $w$-elimination is simpler (the critical-point equation is still a cubic in~$t$, but the endpoint normalization $f(0)=1$ eliminates one variable immediately). The resultant chain involves a $6\times 8$ resultant in~$t_2$ followed by clearing $s^3=2$, giving raw degree $\sim 144$.

\item \textbf{Degree-3 equiripple minimal polynomial.} For degree-3 polynomial approximation of $\sqrt{1+t^2}$, the parameter $t_1\approx 0.10099$ is computed to 550~digits. A plausible path would be to eliminate an 8-variable system (three critical points $t_1, t_2, t_3$, three auxiliary radicals $w_i = \sqrt{1+t_i^2}$, the coefficient~$d$, and $s=\sqrt{2}$). The B\'ezout bound is in the thousands; the irreducible factor likely has degree 200--1000. The main obstacle seems to be that the resultant chain involves at least 5 elimination steps with intermediate degrees in the hundreds. Therefore memory and time requirements are likely the main limitations.

\item \textbf{Taking advantage of parallelism.} 

All timings assume x86-64, Intel Golden Cove / AMD Zen 4, scalar single-precision (Table \ref{timing}).
The naive way to compute $h(x,y)$ consists in respecting the order of operations given in this paper, which requires $22$ cycles. However, by exploiting parallelism (Table~\ref{par}) one can complete the calculation in $16$ cycles.

We need to compute

\[
h(x,y) = ax+(b+c)y+\frac{cy^2}{x}.
\]

\begin{table}[t]
\centering
\small
\caption{Parallel schedule for the degree-2 evaluation $h(x,y)=ax+(b+c)y+cy^2/x$ on an x86-64 core. Each row lists the issue cycle and the independent operations dispatched at that cycle; the division (10~cycles) dominates the critical path, with the dependent multiply and additions overlapping it. The result is ready after 16~cycles.}
\label{par}
\begin{tabular}{p{1.5cm}l}
\toprule
Cycle & Operations dispatched  \\
\midrule
0   & $c\textcolor{red}{/}x$, $a\textcolor{red}{\times}x$, $(b+c)\textcolor{red}{\times}y$, $y\textcolor{red}{\times}y$  \\
3   & $(ax) \textcolor{red}{+}((b+c)y)$  \\
10  & $(c/x)\textcolor{red}{\times} (y^2)$\\
13  & $(ax+(b+c)y)\textcolor{red}{+}(cy^2/x)$    \\
16  & Result ready.\\
\bottomrule
\end{tabular}
\end{table}

\begin{table}[!tb]
\centering
\small
\caption{Latency and reciprocal throughput, in clock cycles, of the scalar single-precision x86-64 instructions used in the evaluation (Intel Golden Cove / AMD Zen~4). The single division (\texttt{divss}) is both high-latency and poorly pipelined, which is why it sets the critical path in Tables~\ref{par} and~\ref{deg4sched}.}
\label{timing}
\begin{tabular*}{\textwidth}{@{\extracolsep{\fill}}lcc@{}}
\toprule
Instruction & Latency (cycles) & Throughput (cycles) \\
\midrule
\texttt{mulss} & 3 & 0.5 \\
\texttt{addss} & 3 & 0.5 \\
\texttt{divss} & 10 & 10 \\
\bottomrule
\end{tabular*}
\end{table}

 We note that we still have room between cycles $6$ and $10$ to perform other operations. An interesting research direction would be to implement higher-order approximations by taking advantage of the room available during divisions to prepare, as efficiently as possible, upcoming multiplications and additions. Let us consider a practical example: moving to degree $4$.

We need to compute

\[
h(x,y) = ax+by+\frac{cy^2}{x}+\frac{dy^3}{x^2}+\frac{ey^4}{x^3}.
\]

This is doable in $22$ cycles (Table~\ref{deg4sched}) which is $37.5\%$ longer than the $16$ cycles of Table~\ref{par}. However this buys roughly $5.5$ bits of accuracy: from $0.4\%$ to a $0.009\%$ error which is a very appreciable payback for only $22-16=6$ extra cycles.

\begin{table}[t]
\centering
\small
\caption{Parallel schedule for the degree-4 evaluation $h(x,y)=ax+by+cy^2/x+dy^3/x^2+ey^4/x^3$ on an x86-64 core. The two dependent divisions ($c/x$ and $y/x$, the latter reused to build $y^2/x^2$ and $y^3/x^3$) lengthen the critical path to 22~cycles, only 6~cycles more than the degree-2 schedule of Table~\ref{par} while gaining roughly 5.5~bits of accuracy.}
\label{deg4sched}
\begin{tabular}{p{1.5cm}l}
\toprule
Cycle & Operations dispatched  \\
\midrule
0  & $a \textcolor{red}{\times} x$, $(b+c)\textcolor{red}{\times}y$, $y\textcolor{red}{\times}y$, $c\textcolor{red}{/}x$, $y\textcolor{red}{/}x$, $d \textcolor{red}{\times} y$, $e \textcolor{red}{\times} y$ \\
3  & $(ax)\textcolor{red}{+}(by)$\\
10 & $(y^2)\textcolor{red}{\times}(c/x)$, $(y/x)\textcolor{red}{\times}(y/x)$, $(y/x)\textcolor{red}{\times}(ey)$\\
13 & $(ax+by)\textcolor{red}{+}(cy^2/x)$, $(dy)\textcolor{red}{\times}(y^2/x^2)$, $(y^2/x^2)\textcolor{red}{\times}(ey^2/x)$\\
16 & $(ax+by+(cy^2/x))\textcolor{red}{+} (dy^3/x^2)$\\
19 & $(ax+by+cy^2/x+dy^3/x^2)\textcolor{red}{+}(ey^4/x^3)$\\
22 & Result ready.\\
\bottomrule
\end{tabular}
\end{table}

\end{enumerate}

\subsection*{Structural questions}

\begin{enumerate}[resume]
\item \textbf{Degree formula.} Find a formula or asymptotic for $d_p$ (the degree of the irreducible factor over $\Q(2^{1/p})$) as a function of~$p$. Current data: $d_2=10$, $d_3=82$. The $t$-resultant input degrees are $(2p-1)\times\max(2p,6)$, and each subsequent elimination multiplies by~$p$ (from the degree-$p$ radical constraints). This gives a total resultant degree growing as $O(p^3)$ or faster. The ratio of irreducible factor to total resultant is $20/160 = 1/8$ for $L_2$ and $246/1008 \approx 1/4$ for $L_3$, suggesting the \enquote{spurious fraction} decreases with~$p$.

\item \textbf{The thin set in Hilbert's theorem.} Theorem~\ref{thm:hilbert} states that non-solvability holds for all $\lambda\in\Q$ outside a thin set~$T$. Characterize~$T$: does it contain any nontrivial rational points, or is the statement effectively ``for all $\lambda > 0$''? At $\lambda = 0$ the problem degenerates (the function becomes constant). For $\lambda < 0$ the function $\sqrt{1+\lambda t^2}$ has a real singularity in $[0,1]$ when $\lambda < -1$, changing the problem qualitatively. For $\lambda > 0$ rational, we expect $\Gal = S_{10}\times C_2$ generically, but cannot rule out finitely many exceptions where an accidental factorization occurs.
 A plausible approach is to compute $\Gal(P_\lambda/\Q)$ at several small rational $\lambda$ (e.g., $\lambda = 1/2, 2, 3, 4$). If all give $S_{10}\times C_2$, this provides evidence that $T\cap\Q_{>0} = \emptyset$.

\item \textbf{Non-solvability for degree $n\ge 3$.} The coupling argument (Proposition~\ref{prop:coupling}) explains why non-solvability appears at $n=2$, and there is no structural reason for it to disappear at higher degree. Prove that the degree-$n$ equiripple constants are non-solvable for all $n\ge 2$. For each~$n$, the equiripple system has $2n+2$ equations in $2n+2$ unknowns. After elimination, one obtains a univariate polynomial in~$t_1$. The Galois group should be a full symmetric group (by Harris's monodromy theorem~\cite{harris1979}), but proving this requires either: (a)~computing the elimination polynomial explicitly (infeasible for $n\ge 3$, see problem~3), or (b)~showing that the monodromy representation is primitive and contains a transposition (which would follow from simple ramification at the branch points). However, without the explicit polynomial, Jordan's theorem cannot be applied directly. A purely structural proof would require showing that the elimination polynomial has no hidden symmetry for any~$n$, which appears to require new techniques.

\item \textbf{Monodromy branch points.} The branched cover $\pi\colon C\to\mathbb{A}^1_\varepsilon$ has branch points at degenerate equiripple configurations. Compute the branch point values of~$\varepsilon$ explicitly and verify that the ramification is simple. The degenerate configurations are: (a) $t_1\to 0$ (first critical point hits the boundary), giving $\varepsilon = (3-2\sqrt{2})/4\approx 0.0429$; (b) $t_1 = t_2$ (critical points collide); (c) $t_2\to 1$ (second critical point hits the boundary). The physical equiripple solution has $\varepsilon\approx 0.00404$, lying between two branch points. Simple ramification would mean exactly two roots of $P_\varepsilon$ coalesce at each branch point, confirming Harris's hypotheses. However, the collision branch point (b) requires solving a degenerate system ($d'(t^*)=0$ and $d''(t^*)=0$ simultaneously), a computation we have not completed.
\end{enumerate}

\subsection*{Extensions}

\begin{enumerate}[resume]
\item \textbf{Higher-$k$ breakpoint degrees.} For degree-1 piecewise with $k=2$, the optimal breakpoint has degree~16 and Galois group $S_8\times C_2$ (Propositions~\ref{prop:breakpoint}--\ref{prop:breakpoint_galois}). What is the algebraic degree for $k=3$, and does the Galois group remain a full symmetric group? A plausible angle could be: for $k=3$, two breakpoints $b_1 < b_2$ must satisfy $E_1^*([0,b_1]) = E_1^*([b_1,b_2]) = E_1^*([b_2,1])$. This is a system of 2 algebraic equations in 2 unknowns after algebraization, and the degree should be accessible by resultant elimination.
\end{enumerate}

\section{Reproducibility}\label{sec:reproducibility}

All computations are reproducible. A PARI/GP script verifying the polynomial, irreducibility, factorization over $\Q(\sqrt{2})$, the Frobenius at $p=17$, and the negative discriminant runs in under 5~seconds. SageMath code for the full resultant elimination chain is available as supplementary material accompanying this paper and will be deposited in a public repository upon publication.

\section*{Conclusion}

The optimal coefficients for low-degree polynomial approximation of $\sqrt{1+t^2}$ exhibit a sharp algebraic phase transition. At degree~1, the equiripple system decouples and produces constants that lie in a degree-4 solvable extension of~$\Q$. At degree~2, the system becomes irreducibly coupled, and the resulting constants satisfy irreducible polynomials of degree~12 or~20 whose Galois groups ($S_{12}$ and $S_{10}\times C_2$) are not solvable. No amount of algebraic ingenuity can produce these constants by radicals.

This phenomenon is not specific to the Euclidean norm. It persists across all $L_p$ norms (with rapidly increasing polynomial degree) and holds generically for families of algebraic target functions. The underlying mechanism is structural: the moment an optimization problem couples multiple parameters through nonlinear constraints, the resulting elimination polynomial is generically as algebraically complex as it can be.

These results establish a previously unobserved connection between two classical areas of mathematics: Chebyshev approximation theory and the Galois theory of algebraic equations.


\printbibliography
\appendix

\section{Values of constants for various degree approximations}
For a degree-$n$ approximation we write the homogeneous form
\begin{align*}
h_n(x,y) = x\,p_n\!\left(\tfrac{y}{x}\right),\qquad 
p_n(t) = \sum_{j=0}^{n} a_{n,j}\,t^{j},
\end{align*}
so that $a_{n,j}$ is the coefficient of $t^{j}$ in the optimal degree-$n$ polynomial $p_n(t)\approx\sqrt{1+t^2}$ on $[0,1]$. Table~\ref{tab:constants} lists these monomial coefficients to 20~decimal places for both error criteria: the \emph{absolute} equiripple (minimizing $\max_t|p_n(t)-\sqrt{1+t^2}|$) and the \emph{relative} equiripple (minimizing $\max_t|p_n(t)/\sqrt{1+t^2}-1|$). The coefficients were obtained by the Remez exchange algorithm carried out in 60-digit arithmetic; in every case the computed polynomial equioscillates at the required $n+2$ points and attains the same maximum error at no other point, certifying optimality. The degree-1 and degree-2 rows agree with the closed forms derived in Sections~\ref{sec:deg1} and~\ref{sec:main}, and the absolute errors match Table~\ref{tab:efficiency}.

\small
\begin{longtable}{@{\extracolsep{\fill}}lll@{}}
\caption{Optimal monomial coefficients $a_{n,j}$ of the degree-$n$ minimax polynomial $p_n(t)=\sum_{j=0}^n a_{n,j}\,t^j$ approximating $\sqrt{1+t^2}$ on $[0,1]$, for $n=1,\ldots,8$ and to 20~decimal places, under the absolute and relative equiripple criteria.}
\label{tab:constants}\\
\toprule
Coefficient & Absolute equiripple & Relative equiripple  \\
\midrule
\endfirsthead
\caption[]{Optimal monomial coefficients $a_{n,j}$ (continued).}\\
\toprule
Coefficient & Absolute equiripple & Relative equiripple  \\
\midrule
\endhead
\midrule
\multicolumn{3}{r@{}}{\textit{continued on next page}}\\
\endfoot
\bottomrule
\endlastfoot
$a_{1,0}$ & $\phantom{-}0.95508986056222734130$ & $\phantom{-}0.96043387010341996525$ \\
$a_{1,1}$ & $\phantom{-}0.41421356237309504880$ & $\phantom{-}0.39782473475931601382$ \\\midrule
$a_{2,0}$ & $\phantom{-}0.99595820792792624312$ & $\phantom{-}0.99650308101017055306$ \\
$a_{2,1}$ & $\phantom{-}0.06644256731401868891$ & $\phantom{-}0.06109893941818041232$ \\
$a_{2,2}$ & $\phantom{-}0.35585457920322387365$ & $\phantom{-}0.36155693220668091052$ \\\midrule
$a_{3,0}$ & $\phantom{-}1.00012795658447520881$ & $\phantom{-}1.00012122881443169729$ \\
$a_{3,1}$ & $          -0.00619498587275529769$ & $          -0.00600202337439002681$ \\
$a_{3,2}$ & $\phantom{-}0.54786400641793172111$ & $\phantom{-}0.54719029451011597413$ \\
$a_{3,3}$ & $          -0.12745545817208137461$ & $          -0.12692449414354287831$ \\\midrule
$a_{4,0}$ & $\phantom{-}1.00008553936375914880$ & $\phantom{-}1.00007569068070087741$ \\
$a_{4,1}$ & $          -0.00451756363033955934$ & $          -0.00413859075376873774$ \\
$a_{4,2}$ & $\phantom{-}0.53800868308049013660$ & $\phantom{-}0.53589726093874930308$ \\
$a_{4,3}$ & $          -0.10958709590538596023$ & $          -0.10596161474822152395$ \\
$a_{4,4}$ & $          -0.00986153989918786583$ & $          -0.01176622653155730232$ \\\midrule
$a_{5,0}$ & $\phantom{-}1.00000989644626306332$ & $\phantom{-}1.00000846540902162006$ \\
$a_{5,1}$ & $          -0.00062006257702921761$ & $          -0.00054378248572955310$ \\
$a_{5,2}$ & $\phantom{-}0.50606162525110040603$ & $\phantom{-}0.50542449786416140566$ \\
$a_{5,3}$ & $          -0.01842310746321587307$ & $          -0.01658227854945195295$ \\
$a_{5,4}$ & $          -0.11561656351328177078$ & $          -0.11777107067761255706$ \\
$a_{5,5}$ & $\phantom{-}0.04281167067552150424$ & $\phantom{-}0.04368970270895549684$ \\\midrule
$a_{6,0}$ & $\phantom{-}0.99999894656138555122$ & $\phantom{-}0.99999902699879887785$ \\
$a_{6,1}$ & $\phantom{-}0.00012187185143533507$ & $\phantom{-}0.00011472116699129120$ \\
$a_{6,2}$ & $\phantom{-}0.49769161735221235575$ & $\phantom{-}0.49778877396920703939$ \\
$a_{6,3}$ & $\phantom{-}0.01653410625732400522$ & $\phantom{-}0.01606897107608885917$ \\
$a_{6,4}$ & $          -0.18204775062004688422$ & $          -0.18106636685846837585$ \\
$a_{6,5}$ & $\phantom{-}0.10123018191063006826$ & $\phantom{-}0.10029282017426146797$ \\
$a_{6,6}$ & $          -0.01931435750123093371$ & $          -0.01898300812228927868$ \\\midrule
$a_{7,0}$ & $\phantom{-}0.99999958015691108921$ & $\phantom{-}0.99999963050530121161$ \\
$a_{7,1}$ & $\phantom{-}0.00005432166802212748$ & $\phantom{-}0.00004885757389857460$ \\
$a_{7,2}$ & $\phantom{-}0.49885052984217509343$ & $\phantom{-}0.49894473712067124494$ \\
$a_{7,3}$ & $\phantom{-}0.00916524012546229683$ & $\phantom{-}0.00856667553578175083$ \\
$a_{7,4}$ & $          -0.15994395826285588701$ & $          -0.15815447506718354700$ \\
$a_{7,5}$ & $\phantom{-}0.06744707441536141232$ & $\phantom{-}0.06472424571200200795$ \\
$a_{7,6}$ & $\phantom{-}0.00613697548178642465$ & $\phantom{-}0.00817832081585720935$ \\
$a_{7,7}$ & $          -0.00749662089685641891$ & $          -0.00809495236764765499$ \\\midrule
$a_{8,0}$ & $\phantom{-}0.99999996323398694229$ & $\phantom{-}0.99999996915431988924$ \\
$a_{8,1}$ & $\phantom{-}0.00000474208412064818$ & $\phantom{-}0.00000398233817020808$ \\
$a_{8,2}$ & $\phantom{-}0.49990296018540403901$ & $\phantom{-}0.49991898296737643744$ \\
$a_{8,3}$ & $\phantom{-}0.00066172443128761477$ & $\phantom{-}0.00053283423215312832$ \\
$a_{8,4}$ & $          -0.12623687146969822443$ & $          -0.12572752823225463529$ \\
$a_{8,5}$ & $          -0.00503812846135399869$ & $          -0.00613105966469614634$ \\
$a_{8,6}$ & $\phantom{-}0.09241268451634408506$ & $\phantom{-}0.09371149116865559082$ \\
$a_{8,7}$ & $          -0.06092694367631230666$ & $          -0.06173034736513950102$ \\
$a_{8,8}$ & $\phantom{-}0.01343346829532930699$ & $\phantom{-}0.01363528139688923081$ \\
\end{longtable}
\normalsize

\end{document}